\documentclass[leqno]{article}

\usepackage{amsmath,graphics,amsfonts,amsthm,amssymb,color}
\usepackage{xypic}
\usepackage{hyperref}

\theoremstyle{plain}
\newtheorem{theorem}{Theorem}[section]
\newtheorem{lemma}[theorem]{Lemma}

\theoremstyle{definition}
\newtheorem{remark}[theorem]{Remark}

\DeclareMathOperator{\ch}{ch}
\DeclareMathOperator{\cl}{cl}
\DeclareMathOperator{\id}{id}
\DeclareMathOperator{\td}{td}
\DeclareMathOperator{\even}{even}
\DeclareMathOperator{\odd}{odd}
\DeclareMathOperator{\Tr}{Tr}
\def\tH{\widetilde{H}}
\def\a{\alpha}

\def\F{\mathbb{F}}
\def\O{\mathcal{O}}
\def\Ql{\mathbb{Q}_{\ell}}

\begin{document}
\author{Katrina Honigs}
\title{Derived equivalent surfaces and abelian varieties, and their zeta functions}

\maketitle

\begin{abstract}
In this paper, it is demonstrated that 
derived equivalence between smooth, projective varieties that are either surfaces or abelian implies equality of zeta functions.
\end{abstract}

Since their definition in the 1960's by Verdier and Grothendieck \cite{verdier} to provide a foundation for homological algebra, derived categories have been regarded with progressively more interest. In particular, the bounded derived category $D^b(X)$ of coherent sheaves on a variety $X$ has interest in its own right as an invariant, which is the subject of this paper.

Bondal and Orlov \cite[Theorem 2.5]{ample} demonstrated that if there is an exact equivalence $D^b(X)\simeq D^b(Y)$ between 
smooth varieties $X$ and $Y$, and $X$ is projective and has ample or anti-ample canonical bundle,
then $X$ is isomorphic to $Y$. 
However, if we 
remove the assumption on the canonical bundle, 
the result no longer holds in general: an abelian variety is always derived equivalent to its dual (via, for instance, the Fourier--Mukai transform with kernel the Poincar\'{e} bundle, as shown in Theorem 2.2 of Mukai's influential work \cite{mukai}), but not always isomorphic to it. 
However, an abelian variety and its dual are isogenous, and furthermore, 
a conjecture of Orlov (\cite[Conjecture 1]{motives}) states that derived equivalent smooth projective varieties have isomorphic motives. This conjecture predicts that derived equivalent smooth projective varieties over finite fields should have equal zeta functions, and in this paper we show this prediction holds for derived equivalent varieties that are either surfaces or abelian.

We will prove in Section \ref{abelian} that any abelian varieties over finite fields that are derived equivalent are also isogenous, or equivalently, by Tate \cite[Theorem 1]{tate}, have equal zeta functions (Theorem \ref{abelian.thm}).
We will show that derived equivalent surfaces have equal zeta functions (Theorem \ref{thmsurfaces})
via a generalization of the argument that Olsson and Lieblich used in 
\cite{LO} for K3 surfaces.
We first develop some necessary background for this proof in Section \ref{mhs}, where we define, for any Weil cohomology, the even and odd Mukai--Hodge structures of an arbitrary smooth, projective variety $X$ and show that these structures are invariant under Fourier--Mukai equivalence. Then we give the proof of Theorem \ref{thmsurfaces} in Section \ref{sec.pf}.

\section{Terminology and preliminaries}

%This paper mostly uses the language of schemes rather than varieties, but where the word \textit{variety} is used, we take it to be an integral, separated scheme of finite type.
%may need to change this definition

We denote the bounded derived category of coherent sheaves on a variety $X$ by $D^b(X)$.
Given two varieties $X$ and $Y$ defined over a field $k$, 
a \textit{derived equivalence} between them is a $k$-linear exact equivalence  $D^b(X)\simeq D^b(Y)$ (see \cite[Definition 3.2]{huybrechts}). 

A functor $F$ between derived categories $D^b(X)$ and $D^b(Y)$ is a \textit{Fourier--Mukai transform} if there exists an object $P \in D^b(X\times Y)$, called a \textit{Fourier--Mukai kernel}, such that $F={p_2}_*(p_1^*(-)\otimes P)$, where pushforward, pullback, and tensor are all in their derived versions. The Fourier--Mukai transform with kernel $P$ is denoted $\Phi_P$.
%In dealing with derived equivalences, we will limit ourselves \textit{Fourier--Mukai equivalences}, that is, to those given by Fourier--Mukai transforms, since,
By a result of Orlov \cite[Theorem 3.2.1]{Ocoh}, given a derived equivalence $F: D^b(X) \to D^b(Y)$ between two smooth, projective varieties $X$ and $Y$, there is an object $P\in D^b(X\times Y)$ unique up to isomorphism such that $\Phi_P$ is naturally isomorphic to $F$ (see Section 5.1 and specifically Corollary 5.17 in \cite{huybrechts}).
Thus, in proving Theorem \ref{thmsurfaces}, we may immediately replace ``derived equivalent'' with ``Fourier--Mukai equivalent''.

\section{Derived equivalent abelian varieties}\label{abelian}

We first prove a result that holds in arbitrary characteristic.

\begin{lemma}\label{isogeny}
Let $A$ and $B$ be abelian varieties defined over an arbitrary field. If $A$ and $B$ are derived equivalent, then $A$ and $B$ are isogenous.
\end{lemma}

\begin{proof}
By Orlov \cite[Theorem 2.19]{abelian}, if $A$ and $B$ are derived equivalent abelian varieties, then $A\times \hat{A} \cong B\times\hat{B}$. 

By Poincar\'{e}'s complete reducibility theorem, abelian varieties decompose uniquely, up to isogeny, into products of simple abelian varieties (see Corollary 1, page 174 of \cite{mumford}). Then, since any abelian variety is isogenous to its dual, $A\times \hat{A} \cong B\times\hat{B}$ implies that $A$ is isogenous to $B$.
\end{proof}

\begin{theorem}\label{abelian.thm}
Let $A$ and $B$ be abelian varieties defined over a finite field $\F$. If $A$ and $B$ are derived equivalent, then $A$ and $B$ have equal zeta functions.
\end{theorem}

\begin{proof}
Let $A$ and $B$ be derived equivalent abelian varieties over a finite field $\F$.
By Lemma \ref{isogeny}, $A$ and $B$ are isogenous.
Isogenous abelian varieties over a finite field have equal zeta functions (see Tate \cite[Theorem 1]{tate}, for example, for a proof). Hence, $A$ and $B$ have equal zeta functions.
\end{proof}

\section{Mukai--Hodge structures}\label{mhs}

Let $H$ be an arbitrary Weil cohomology with coefficients in a characteristic 0 field $K$. A classic reference on Weil cohomologies is Kleiman's article \cite{kleiman};  however, Kleiman's definition of a Weil cohomology does not include Tate twists, which we need to take into account. 
So, we define a Weil cohomology theory to be a Poincar\'e duality theory with supports, as defined in Bloch and Ogus  \cite{blochogus}. We will denote the $i^{\rm th}$ cohomology group of $X$ twisted by $n$ as $H^i(X)(n)$.
In this paper, we take the \textit{even and odd Mukai--Hodge structures} of a dimension $d_X$ smooth, projective variety $X$ to be the pure-weight (weights 0 and 1, respectively) sums of cohomology groups given by:
\begin{align*}
\tH^{\even}(X/K)&=\bigoplus^{d_X}_{i=0} H^{2i}(X/K)(i), \\
\tH^{\odd}(X/K)&=\bigoplus^{d_X}_{i=1} H^{2i-1}(X/K)(i).
\end{align*}

Given varieties $X$ and $Y$ (of dimensions $d_X$ and $d_Y$), a Fourier--Mukai transform $\Phi_P: D^b(X) \to D^b(Y)$ with kernel $P\in D^b(X\times Y)$ induces 
the operation 
$\Psi_P^{CH}=p_{2*}(v(P)\cup p_1^*(-))$
on Chow groups where $v(P):= \ch (P).\sqrt{\td(X\times Y)}$ is the Mukai vector of $P$. 
Then we can take the cycle class of $v(P)$ inside any Weil cohomology theory of our choice to induce a map $\Psi_P=p_{2*}(\cl(v(P))\cup p_1^*(-))$ on cohomology.
More specifically, we define
%via the map ${p_2}_*(p_1^*(-)\cup v(P))$
\[\displaylines{
\Psi_P^{i,j}: H^i(X/K) \xrightarrow{p_1^*} H^i(X\times Y/K) \xrightarrow{\cup v^{j}(P)} H^{i+2j}(X\times Y /K)(j)
\hfill\cr\hfill
\xrightarrow{p_{2*}} H^{i+2(j-d_X)}(Y/K)(j-d_X),
}\]
for $0\leq i\leq 2d_X$ and $0\leq j\leq 2d_Y$, where $v^j(P)$ is the degree $2j$ part of $v(P)$. We denote the map $\Psi^{i,j}_P$ induces on the $i^{\rm th}$ cohomology twisted by $l$ as $\Psi^{i,j}_P(l): H^i(X/K)(l) \to H^{i+2j-2d_X}(Y/K)(l+j-d_X)$.

Observe that the maps
\begin{align*}
\Psi_P^{\even}:=\bigoplus_{i=0}^{d_X} \sum_{j=0}^{d_X+d_Y} \Psi^{2i,j}_P(i-d): 
\tH^{\even}(X/K) &\to \tH^{\even}(Y/K), \\
\Psi_P^{\odd}:=\bigoplus_{i=1}^{d_X} \sum_{j=0}^{d_X+d_Y} \Psi^{2i-1,j}_P(i-d): 
\tH^{\odd}(X/K) &\to \tH^{\odd}(Y/K),
\end{align*}
are well defined. 

\begin{lemma}\label{lemma}
Given smooth, projective varieties $X$ and $Y$ and a Fourier--Mukai equivalence $\Phi_P: D^b(X)\to D^b(Y)$, the maps $\Psi_P^{\even}$ and $\Psi_P^{\odd}$ are isomorphisms.
\end{lemma}

\begin{proof}
Since Fourier--Mukai transforms between smooth, projective varieties have left and right adjoints that are also Fourier--Mukai transforms (see \cite[Propostion 5.9]{huybrechts}) there is a $P'\in D^b(X\times Y)$ such that $\Phi_{P'}$ is quasi-inverse to $\Phi_P$.
%(\cite{mukai}; see also \cite[Propostion 5.9]{huybrechts}). 
Since the compositions $\Phi_{P'}\circ\Phi_P\cong \id_{D^b(Y)}$ and $\Phi_P\circ \Phi_{P'}\cong\id_{D^b(X)}$ are fully faithful and exact, 
$\O_{\Delta}\in D^b(X\times X)$ and $\O_{\Delta}\in D^b(Y\times Y)$ are the unique (up to isomorphism) objects such that $\Phi_{P'}\circ\Phi_P\cong \Phi_{\O_{\Delta}}$ and 
$\Phi_P\circ \Phi_{P'}\cong\Phi_{\O_{\Delta}}$ (Orlov \cite{Ok3,Ocoh}; see also \cite[Theorem 5.14]{huybrechts}).

Hence, in order to show that $\Psi_P^{\even}$ and $\Psi_P^{\odd}$ are isomorphisms, it suffices to prove the following two statements.
\begin{enumerate}
\item[(1)] For $Q\in D^b(X\times Y)$, $R\in D^b(Y\times Z)$, $S\in D^b(X\times Z)$ such that $\Phi_R$, $\Phi_Q$ and $\Phi_S$ are equivalences and $\Phi_R\circ\Phi_Q\cong \Phi_S$, we have $\Psi^{\even}_R\circ\Psi^{\even}_Q\cong \Psi^{\even}_S$ and $\Psi^{\odd}_R\circ\Psi^{\odd}_Q\cong \Psi^{\odd}_S$.
\item[(2)] $\Psi^{\even}_{\O_{\Delta}}$ and $\Psi^{\odd}_{\O_{\Delta}}$ act identically.
\end{enumerate}

(1)\ By Mukai \cite[Proposition 1.3]{mukai}, $\Phi_R\circ\Phi_Q\cong\Phi_{S'}$ for $S'={\pi_{XZ}}_*({\pi_{XY}}_* R \otimes {\pi_{YZ}}_* Q)$, where $\pi_{XY}$, $\pi_{YZ}$, and $\pi_{XZ}$ are the projection maps from $X\times Y\times Z$ to $X\times Y$, $Y\times Z$, and $X\times Z$. Since $\Phi_{S'}$ and $\Phi_{S}$ are equivalences by Orlov \cite[Theorem 2.2]{Ok3}, $S\cong S'$, and, without loss of generality, we may let $S=S'$. 

Mukai's argument shows directly that $\Phi_R\circ\Phi_Q$ and $\Phi_S$ have isomorphic kernels (and so are naturally isomorphic derived functors) using the projection formula and the flat base change theorem. The same arguments can be applied inside the Chow ring to show that  $\Psi_R^{CH}\circ\Psi_Q^{CH}=\Psi_S^{CH}$, and they still apply once we descend to a Weil cohomology theory of our choice by taking cycle classes. We note that Huybrechts' proof of a result analogous to \cite[Proposition 1.3]{mukai} for realizations of the Fourier--Mukai functor acting on the cohomology $H^*(X,\mathbb{Q})$ of the constant sheaf $\mathbb{Q}$ on complex manifolds $X$ \cite[Lemma 5.32]{huybrechts} rests on this same argument -- showing Mukai's proof still works after descending to this particular cohomology theory.

(2)\ As shown in the proof of \cite[Proposition 5.33]{huybrechts}, as a direct consequence of the Grothendieck--Riemann--Roch formula, for any smooth, projective variety $X$, the operation $p_{1*}(p_2^*(-) \cup v(\O_{\Delta}))$ acts identically on cohomology groups. Since $\Psi^{\even}_{\O_{\Delta}}$ and $\Psi^{\odd}_{\O_{\Delta}}$ are given by the action of $p_{1*}(p_2^*(-) \cup v(\O_{\Delta})$ on $\tH^{\even}(X/K)$ and $\tH^{\odd}(Y/K)$, respectively, they each act identically.
\end{proof}

\begin{remark}
The choice of weight for the even and odd Mukai--Hodge structures was an arbitrary one. We could alter either structure by twisting it by the same amount in each dimension. The maps $\Psi$, also twisted by that same amount, would still induce isomorphisms on the even and odd Mukai--Hodge structures of Fourier--Mukai equivalent varieties.
\end{remark}

\section{Derived equivalent surfaces}\label{sec.pf}

\begin{theorem}\label{thmsurfaces}
Let $X$ and $Y$ be surfaces (i.e., smooth, projective varieties of dimension 2) over a finite field $\F$ 
such that $D^b(X)$ is equivalent to $D^b(Y)$. Then $X$ and $Y$ have the same zeta-function. In particular, $\#X(\F)=\#Y(\F)$.
\end{theorem}

\begin{proof}
Let $\mathbb{F}=\mathbb{F}_q$ a field with $q$ elements,
and $P\in D^b(X\times Y)$  be the  kernel  of a Fourier--Mukai  equivalence $\Phi_P: D^b(X)\stackrel{\sim}{\longrightarrow} D^b(Y)$.

By the Lefschetz fixed-point formula for Weil cohomologies (see Proposition 1.3.6 and Section 4 of 
Kleiman \cite{kleiman}), to prove Theorem \ref{thmsurfaces} it is sufficient to show that, for some Weil cohomology $H$, the traces of the Frobenius map $\varphi$ acting on $H^i(X/K)$ and $H^i(Y/K)$ are the same for $0\leq i \leq 4$. Indeed, this condition is necessary as well: by the ``Riemann hypothesis'' portion of the Weil Conjectures, the characteristic polynomials of Frobenius acting on the $i^{\rm th}$ cohomology groups of smooth, projective varieties with equal zeta functions are equal.

Note that the traces of $\varphi$ acting on $H^i(X/K)$ and $H^i(Y/K)$ are trivially equal for $i=0,4$.

By Lemma \ref{lemma}, $\tH(X/K)^{\even}=\tH(Y/K)^{\even}$ and $\tH(X/K)^{\odd}=\tH(Y/K)^{\odd}$, so 
the trace of the Frobenius acting on both sides of each equation is the same.

Recall that 
\vspace*{-4pt}
\[\Tr(\varphi^*|H^i(X/K)(l))=\frac{1}{q^l}\Tr (\varphi^*|H^i(X/K)).
\vspace*{-3pt}
\]
In the case of the even Mukai--Hodge structure, then, 
\[\displaylines{
\Tr(\varphi^* | \tH(X/K)^{\even})\hfill\cr
\hfill=\Tr(\varphi^* | H^0(X/K)(-2))+\Tr(\varphi^* | H^2(X/K)(-1))+\Tr(\varphi^* | H^4(X/K))\cr
\hfill\rlap{$=2q^2 + \Tr(\varphi^* | H^2(X/K)(-1))$.}\phantom{=\Tr(\varphi^* | H^0(X/K)(-2))+\Tr(\varphi^* | H^2(X/K)(-1))+\Tr(\varphi^* | H^4(X/K))}
}
\]
Hence, $\Tr(\varphi^*| H^2(X/K))=\Tr(\varphi^*| H^2(Y/K))$.

The following lemma is sufficient to complete the proof of the theorem. Observe that until now, we have been working with an arbitrary Weil cohomology $H$. The lemma switches to \'etale cohomology since it uses Deligne's theory of weights in its proof.

\begin{lemma}\label{claim}
Let $V$ be a smooth, projective variety of dimension $d$ over the field $\mathbb{F}_q$ and $H$ be \'etale cohomology with $\Ql$-coefficients for $(\ell,q)=1$. If the eigenvalues (with multiplicity) of $\varphi^*$ acting on $H^i(V/\Ql)$, $0\leq i<\frac{d}{2}$, are $\{\a_1,\ldots,\a_n\}$, then the eigenvalues of $\varphi^*$ acting on $H^{2d-i}(V/\Ql)$ are $\{ q^{d-i}{\a_1},\ldots, q^{d-i}{\a_n} \}$.
\end{lemma}

By this lemma, working now in \'etale cohomology, $\Tr(\varphi^*|H^3(X/\mathbb{Q}_l))=q\Tr(\varphi^*|H^1(X/\Ql))$. Together with the equality of the odd Mukai--Hodge structures of $X$ and $Y$, this shows that $\Tr(\varphi^*|H^1(X/\Ql))=\Tr(\varphi^*|H^1(Y/\Ql))$ and $\Tr(\varphi^*|H^3(X/\Ql))=\Tr(\varphi^*|H^3(Y/\Ql))$, concluding the proof of the theorem.
\end{proof}

\begin{proof}[Proof of Lemma \ref{claim}]
The Hard Lefschetz Theorem for \'etale cohomology (Deligne \cite[Th\'eor\`eme 4.1.1]{weilii}) states that the map $L^{d-i}: H^i(V/\Ql)(i-d) \stackrel{\sim}{\longrightarrow} H^{2d-i}(V/\Ql)$ is an isomorphism, where $L^{d-i}$ is the $(d-i)^{\rm th}$ iteration of the Lefschetz operator $L$, which is given by intersecting with the hyperplane class. Since $L$ commutes with the action of the Frobenius map on cohomology, the Hard Lefschetz Theorem shows that if the eigenvalues of the action of $\varphi^*$ on $H^i(V/\Ql)$ are $\{\a_1,\ldots,\a_n\}$, then the eigenvalues of the action of $\varphi^*$ on $H^{2d-i}(V/\Ql)$ are $\{q^{d-i}\a_1,\ldots,q^{d-i}\a_n\}$.
\end{proof}

\subsection*{Acknowledgements}
Thanks to my advisor, Martin Olsson, for suggesting this research topic, and I'd like to thank both him and Arthur Ogus for many helpful conversations. I would also like to thank the referee for many helpful suggestions, including simplifications to the proofs of Theorem \ref{abelian.thm} and Lemma \ref{claim}.

\bibliographystyle{plain}
\bibliography{biblio}

\end{document}